\documentclass[12pt]{amsart}
\usepackage{a4wide}
\usepackage{amsthm, amssymb,color}
\usepackage{amsfonts, amscd}
\usepackage{epsfig,multicol}
\theoremstyle{plain} \numberwithin{equation}{section}
\newtheorem{theo}{Theorem}[section]
\newtheorem{coro}[theo]{Corollary}
\newtheorem{prop}[theo]{Proposition}
\newtheorem{lemm}[theo]{Lemma}

\theoremstyle{definition}
\newtheorem{defi}[theo]{Definition}
\newtheorem{exam}[theo]{Example}
\newtheorem{rema}[theo]{Remark}

\def\Z{\mathbb Z}
\def\C{\mathbb C}
\def\R{\mathbb R}

\def\fg{g}

\def\G0{G^0}

\def\NG0T{N_{\G0}(T)}

\def\epsilona{\epsilon_\alpha}

\def\g{\mathfrak g}

\def\RG{\Delta(G)}

\def\MV{V}

\def\G0{G}

\def\NG0T{N_{\G0}(T)}

\renewcommand{\H}{\hat H}

\DeclareMathOperator{\rank}{rank}
\DeclareMathOperator{\GL}{GL}
\DeclareMathOperator{\Aut}{Aut}

\DeclareMathOperator{\Symp}{Symp}
\DeclareMathOperator{\Diff}{Diff}

\DeclareMathOperator{\Ker}{Ker}
\DeclareMathOperator{\U}{U}
\DeclareMathOperator{\PU}{PU}
\DeclareMathOperator{\Hom}{Hom}
\DeclareMathOperator{\Fix}{Fix}

\DeclareMathOperator{\SO}{SO}
\DeclareMathOperator{\trace}{trace}

\def\Jst{J_{st}}

\begin{document}
\title{Root systems and symmetries of torus manifolds}
\author[S. Kuroki]{Shintar\^o KUROKI}
\address{Graduate School of Mathematical Sciences, The University of Tokyo, 3-8-1 Komaba, Meguro-ku, Tokyo, 153-8914, Tokyo, Japan}
\email{kuroki@ms.u-tokyo.ac.jp}
\author[M. Masuda]{Mikiya Masuda}
\address{Department of Mathematics, Osaka City
University, Sumiyoshi-ku, Osaka 558-8585, Japan.}
\email{masuda@sci.osaka-cu.ac.jp}

\date{\today}
\thanks{
The first author was partially supported by 
Grant-in-Aid for Scientific Research (S)
24224002, Japan Society for Promotion of Science, and
the JSPS Strategic Young Researcher Overseas Visits Program for Accelerating Brain Circulation
``Deepening and Evolution of Mathematics and Physics, Building of International Network Hub based on OCAMI'',
and the bilateral program between Japan and Russia:
``Topology and geometry of torus actions and combinatorics of orbit quotients''.
The second author was partially supported by Grant-in-Aid for Scientific Research 25400095}

\begin{abstract}
We associate a root system to a finite set in a free abelian group and prove that its irreducible subsystem is of type A, B or D.
Applying this general result to a torus manifold $M$, where a torus manifold is a $2n$-dimensional connected closed smooth manifold with a smooth effective action of an $n$-dimensional compact torus having a fixed point, 
we introduce a root system $R(M)$ for $M$ 
and show that if the torus action on $M$ extends to a smooth action of a connected compact Lie group $G$, then the root system of $G$ is a subsystem of $R(M)$ so that 
any irreducible factor of the Lie algebra of $G$ is of type A, B or D.  
Moreover, we show that only type A appears if $H^*(M)$ is generated by $H^2(M)$ as a ring.
We also discuss a similar problem for a torus manifold with an invariant stable complex structure.  Only type A appears in this case, too.  
\end{abstract}

\maketitle 


\section{Introduction}
\label{intro}

One of the original motivation of toric geometry initiated by Demazure \cite{de70} was to study the automorphism group $\Aut(X)$ of a toric variety $X$. He completely determined $\Aut(X)$ when $X$ is complete and non-singular, by introducing a root system associated to the fan of $X$ (see also \cite{oda88}). Cox \cite{cox95} generalized this result to the case when $X$ is complete and simplicial, i.e., a compact orbifold, using the homogeneous coordinate ring (Cox ring) of $X$.  Note that $\Aut(X)$ is an algebraic group in these cases and the $\C^*$-torus acting on $X$ is a maximal torus of $\Aut(X)$.   

A symplectic toric manifold $M$ is a compact symplectic manifold of dimension $2n$ with an action of an $n$-dimensional compact torus $T$ preserving the symplectic form and admitting a moment map.  As is well-known, $M$ is determined by the moment map image which is a simple polytope satisfying a non-singular condition (\cite{delz88}) and hence $M$ is equivariantly diffeomorphic to the projective non-singular toric variety $X_P$ associated to the normal fan of $P$. Unlike the complex case above, the group $\Symp(M)$ of symplectomorphisms of $M$ is infinite dimensional.  The second named author \cite{masu10} introduced a root system $R(P)$ for a non-singular polytope $P$ following the idea of Demazure and showed that if $M_P$ is the symplectic toric manifold corresponding to $P$ and $G$ is a compact Lie subgroup of $\Symp(M_P)$ containing the torus $T$ acting on $M_P$, then the root system $\Delta(G)$ of $G$ is a subsystem of $R(P)$ and there is a compact Lie subgroup $G_{\rm max}$ of $\Symp(M_P)$ such that $\Delta(G_{\rm max})=R(P)$.  Note that $R(P)$ agrees with the root system of the reductive part of $\Aut(X_P)$ and any irreducible subsystem of $R(P)$ is of type A. 

In this paper we consider a similar problem to the above in the smooth category.  Our geometrical object is a {\it torus manifold} $M$ which is an orientable connected closed smooth manifold of dimension $2n$ with an effective smooth action of the compact torus $T$ of dimension $n$ having a fixed point (\cite{ha-ma03}).  Complete non-singular toric varieties with the restricted $T$-action and symplectic toric manifolds are examples of torus manifolds.  However, there are many more torus manifolds, for instance, the standard sphere $S^{2n}$ of dimension $2n$ is a torus manifold with a natural $T$-action but it is not toric when $n\ge 2$.  

The $T$-action on the complex projective space $\C P^n$ or $S^{2n}$ extends to a transitive action of $\PU(n+1)$ or $\SO(2n+1)$ in a natural way. The first named author \cite{kuro10}, \cite{kuro11} classified torus manifolds when the $T$-actions extend to smooth actions of compact Lie groups which are transitive or have a codimension one principal orbit.  
If the $T$-action on a torus manifold $M$ extends to an effective action of a compact Lie group $G$, then $T$ is a maximal torus of $G$ and we may think of $G$ as a Lie subgroup of the group $\Diff(M)$ of diffeomorphisms of $M$.
The first purpose of this paper is to introduce a root system $R(M)$ for a torus manifold $M$, where $R(M_P)=R(P)$, and prove the following.  

\begin{theo}[see Theorem~\ref{theo:2-1}, Corollary~\ref{coro:3-1} and Theorem~\ref{theo:4-1}] \label{theo:1-1}
Let $M$ be a torus manifold.  Then any irreducible subsystem of $R(M)$ is of type A, B or D and only type A appears if $H^*(M)$ is generated by $H^2(M)$ as a ring. Moreover, if $G$ is a compact Lie subgroup of $\Diff(M)$ containing the torus acting on $M$, then the root system of $G$ is a subsystem of $R(M)$. \end{theo}

Wiemeler \cite{wiem12} proves that if the $T$-action on a torus manifold $M$ extends to an effective smooth action of a compact Lie group $G$, then any simple factor of (the Lie algebra of) $G$ is of type A, B or D.  
The theorem above recovers this result.

For a torus manifold $M$ with an invariant stable complex structure $J$ (such a manifold is called a unitary toric manifold in \cite{masu99}), one can also define a root system, denoted $R(M,J)$, which is a subsystem of $R(M)$.  The subgroup $\Diff(M,J)$ of $\Diff(M)$ preserving the stable complex structure $J$ is much smaller than $\Diff(M)$, in fact, known to be a Lie group (see Proposition~\ref{prop:5-1}). Our second main result in this paper is a stable complex  version of Theorem~\ref{theo:1-1}. 

\begin{theo}[see Theorem \ref{theo:5-1} and Corollary \ref{coro:5-1}] \label{theo:1-2}
Let $M$ be a torus manifold with an invariant stable complex structure $J$. Then any irreducible subsystem of $R(M,J)$ is of type A and if $G$ is a compact Lie subgroup of $\Diff(M,J)$ containing the torus acting on $M$, then the root system of $G$ is a subsystem of $R(M,J)$. 
\end{theo}

It is natural to ask whether there is a compact Lie subgroup $G_J$ of $\Diff(M,J)$ whose root system agrees with $R(M,J)$.  We do not know the answer in general but we can find such $G_J$ in some cases, for instance when $M$ is a complete non-singular toric variety but $J$ is not necessarily the standard complex structure.  The observation also shows that $R(M,J)$ and $G_J$ actually depend on $J$.

We make a remark on our root systems.  The root system introduced by Demazure for a complete non-singular toric variety is defined by the set of primitive edge vectors in the fan associated with the toric variety.  It does not depend on the underlying simplicial complex of the fan.  Similarly, the root system $R(P)$ for a non-singular polytope $P$ is defined by the primitive normal vectors to the facets of $P$ and does not depend on the combinatorial structure of $P$.  This is also the case for our root systems $R(M)$ and $R(M,J)$.  They are defined by some finite set $\{ v_i\}$ in $H_2(BT)=\Hom(S^1,T)$ where $v_i$'s correspond to the primitive edge vectors in the complex toric case and primitive normal vectors in the symplectic toric case.   

The organization of this paper is as follows. In Section \ref{sect:2}, we associate a root system $R(V)$ to a finite set $V$ in a free abelian group and prove that any irreducible subsystem of $R(V)$ is of type A, B or D. In Section \ref{sect:3}, we deduce the finite set $V$ for a torus manifold $M$ using the equivariant cohomology of $M$ and define the root system $R(M)$ to be $R(V)$. In Section \ref{sect:4}, we apply the results in the previous sections to extended $G$-actions on $M$ and complete the proof of Theorem~\ref{theo:1-1}.  In Section \ref{sect:5}, we study the stable complex transformations of a torus manifold $M$ with an invariant stable complex structure $J$.  We define the root system $R(M,J)$ as a subsystem of $R(M)$ and prove Theorem~\ref{theo:1-2}.  In Section \ref{sect:6}, we discuss $R(M,J)$ and $\Diff(M,J)$ for complete non-singular toric varieties $M$ with various stable complex structures $J$.


\section{Root system associated to a finite set} \label{sect:2}

Let $N$ be a free abelian group of rank $n$ and $\MV=\{v_i\}_{i=1}^m$ be a finite set of non-zero elements in $N$ which generates a subgroup of full rank.  In this section we will associate a root system $R(\MV)$ to the finite set $\MV$ and prove that any irreducible subsystem of $R(\MV)$ is of type A, B or D.  
This fact will be applied to the study of symmetry of a torus manifold in later sections.  

\begin{defi} \label{defi:2-1}
Let $\MV=\{v_i\}_{i=1}^m$ be a finite subset of $N$ which generates a subgroup of full rank.  We define $R(\MV)$ to be the subset of $N^*:=\Hom(N,\Z)$ consisting of all elements $\alpha$ of the following either type 1 or type 2: 
\begin{enumerate}
\item[type 1:] $|\langle\alpha,v_i\rangle|=1$ for some $i$ and $\langle\alpha,v_k\rangle=0$ for $k\not=i$,  
\item[type 2:] $|\langle \alpha,v_i\rangle|=|\langle \alpha,v_j\rangle|=1$ for some $i,j$ and $\langle\alpha,v_k\rangle=0$ for $k\not=i,j$,
\end{enumerate}
where $\langle\ ,\ \rangle$ denotes the natural pairing between $N^*$ and $N$.  
\end{defi}

We shall give some examples of $R(\MV)$.  

\begin{exam} \label{exam:2-1}
Take $N=\Z^n$ and let $\{e_i\}_{i=1}^n$ be the standard basis of $\Z^n$ and $\{e_i^*\}_{i=1}^n$ be its dual basis. 
If $\MV=\{e_i\}_{i=1}^n$, then 
\[
R(\MV)=\{\pm e_i^*\ (1\le i\le n),\ \pm e_i^*\pm e_j^*\ (1\le i<j\le n) \}.
\]
This is a root system of type $B_n$ and $\pm e_i^*$ are of type 1 while $\pm e_i^*\pm e_j^*$ are of type 2.  
If $\MV=\{e_1,\dots,e_n, -\sum_{i=1}^ne_i\}$, then 
\[
R(\MV)=\{\pm e_i^*\ (1\le i\le n),\ \pm(e_i^*-e_j^*) \ (1\le i<j\le n)\}.
\]
This is a root system of type $A_n$ and any element in $R(\MV)$ is of type 2. 
\end{exam}

If $R(\MV)$ contains an element $\alpha$ of type 1 (resp. type 2), then all the elements in $\MV$ except one element (resp. two elements) lie in the kernel of $\alpha$ which is a subgroup of corank one.  Noting this, we see the following. 

\begin{exam} \label{exam:2-2}
Let $N$ be a free abelian group of rank $2$ and $\MV=\{v_i\}_{i=1}^m$ $(m\ge 2)$ be a subset of $N$.  Suppose that $v_i$ and $v_{i+1}$ form a basis of $N$ for $i=1,2,\dots,m$, where $v_{m+1}=v_1$.  Then $R(\MV)$ is empty if $m\ge 5$.  Moreover, one can check that $R(\MV)$ is of type $B_2$ if $m=2$, of type $A_2$ if $m=3$ (see Example~\ref{exam:2-1}) and of type $A_1$ or $A_1\times A_1$ if $m=4$.  
\end{exam}

One can interpret $R(\MV)$ as follows.  Assembling evaluation maps $\langle\ ,v_i\rangle\colon N^*\to \Z$, we obtain a homomorphism $f:=\prod_{i=1}^m\langle\ ,v_i\rangle\colon N^*\to \Z^m$.  Then $R(\MV)$ is the inverse image of 
\begin{equation*} \label{eq:2-2}
R(m):=\{ \alpha\in \Z^m\mid (\alpha,\alpha)=1\text{ or }2\}
\end{equation*}
by $f$, where $(\ ,\ )$ denotes the standard scalar product on $\Z^m$.  Since $\MV=\{v_i\}_{i=1}^m$ is assumed to generate a subgroup of full rank, $f$ is injective so that we may identify $R(\MV)$ with its image $f(R(\MV))=f(N^*)\cap R(m)$.  As is well-known, $R(m)$ is a root system of type B (see \cite[p.64]{hump72}) and we will see below that $R(\MV)$ is also a root system.  However, $R(\MV)$ is not necessarily of type B and the main purpose of this section is to determine its type.  

In the following, we think of $N^*$ as a subgroup of $\Z^m$ and $R(\MV)=R(m)\cap N^*$ through the map $f$.  The reflection on $\Z^m$ with respect to $\alpha\in R(m)$ is given by 
\begin{equation} \label{eq:2-3}
r_\alpha(\beta):=\beta-\frac{2(\beta,\alpha)}{(\alpha,\alpha)}\alpha \qquad (\beta\in \Z^m)
\end{equation}
and $r_\alpha$ preserves $R(m)$.  If $\alpha$ is in $N^*$ (and hence $\alpha\in R(\MV)=R(m)\cap N^*$), then $r_\alpha$ also preserves $N^*$ so that it preserves $R(\MV)$.  Since $R(m)$ is a root system, this shows that $R(\MV)$ is also a root system.

We make one remark.  The scalar product $(\ ,\ )$ on $\Z^m$ induces a positive definite symmetric bilinear form on $N^*$ through $f$, denoted also by $(\ ,\ )$.  One notes that 
\begin{equation} \label{eq:2-4}
(\alpha,\beta)=\sum_{i=1}^m \langle \alpha, v_i\rangle\langle \beta,v_i\rangle\quad\text{for $\alpha,\beta\in N^*$}.
\end{equation}

In the rest of this section, we shall investigate irreducible subsystems of the root system $R(\MV)$. For $\alpha,\beta\in R(\MV)$ we set
\begin{equation} \label{eq:2-7}
a_{\beta,\alpha}:=\frac{2(\beta,\alpha)}{(\alpha,\alpha)} 
\end{equation}
so that 
\[
r_\alpha(\beta)=\beta-a_{\beta,\alpha}\alpha
\]
by \eqref{eq:2-3}. Note that $a_{\beta,\alpha}$ is either $0$, $\pm 1$ or $\pm 2$.  
We say that $\alpha$ is \emph{joined} to $\beta$ when $a_{\beta,\alpha}\not=0$.  As is well-known, when both $\alpha$ and $\beta$ are simple roots, $a_{\beta,\alpha}\le 0$, and they are joined by an edge in the Dynkin diagram of an irreducible subsystem of $R(\MV)$ if and only if $a_{\beta,\alpha}<0$.  

We will denote the standard basis of $\Z^m$ by $\{e_i^*\}_{i=1}^m$.  Note that a type 1 (resp. type 2) element of $R(\MV)$ is of the form $\pm e_i^*$ (resp. $\pm e_i^*\pm e_j^*$ $(i\not=j)$).      

\begin{defi} \label{defi:2-2}
We say that type 2 elements in $R(\MV)$ are \emph{conjugate} if they are of the form $\pm e_i^*\pm e_j^*$ for some common $i\not=j$ and linearly independent.  For instance, $e_i^*+e_j^*$ and $e_i^*-e_j^*$ are conjugate.
\end{defi}

The following lemma will be used to prove Theorem \ref{theo:2-1} stated below.  

\begin{lemm} \label{lemm:2-2}
Let $\Phi$ be an irreducible subsystem of $R(\MV)$.
\begin{enumerate}
\item Let $\alpha, \beta,\gamma,\delta$ be simple roots in $\Phi$ and of type 2.  If $a_{\beta,\alpha}=a_{\gamma,\alpha}=a_{\delta,\alpha}=-1$, then there exists a conjugate pair in $\{\beta,\gamma,\delta\}$.
\item Let $\beta$ and $\gamma$ be conjugate simple roots of type 2 in $\Phi$ and $\lambda$ be another type 2 simple root in $\Phi$.  If $a_{\beta,\lambda}=-1$, then $a_{\gamma,\lambda}=-1$.  
\end{enumerate}
\end{lemm}

\begin{proof}
(1) Since $\alpha,\beta,\gamma$ are of type 2,  the assumption means that 
\begin{equation} \label{eq:2-9}
(\beta,\alpha)=(\gamma,\alpha)=(\delta,\alpha)=-1
\end{equation} 
by \eqref{eq:2-7}.  Recall that type 2 elements are of the form $\pm e_i^*\pm e_j^*$.  Write $\alpha=ae_i^*-b e_j^*$ with $a,b\in\{\pm 1\}$.  Then \eqref{eq:2-9} implies that two of $\beta, \gamma, \delta$, say $\beta$ and $\gamma$, must be of the form $b e_j^*+c e_k^*$ and $b e_j^*+d e_\ell^*$ for some $k, \ell\not=i$ and $c,d \in\{\pm 1\}$.  If $k\not=\ell$, then $a_{\beta,\gamma}=(\beta,\gamma)=1$ but $a_{\beta,\gamma}\le 0$ because $\beta$ and $\gamma$ are simple roots.  This shows that $k=\ell$, proving that $\beta$ and $\gamma$ are conjugate. 

(2) The assumption $a_{\beta,\lambda}=-1$ implies $(\beta,\lambda)<0$.  Moreover, $(\gamma,\lambda)\le 0$ since $\gamma$ and $\lambda$ are simple roots.  On the other hand, since $\beta$ and $\gamma$ are conjugate, we may assume $\beta=b e_j^*+c e_k^*$ and $\gamma=b e_j^*-c e_k^*$ for some $j\not=k$ and $b,c\in\{\pm 1\}$.  These show that $\lambda$ must be of the form $-b e_j^*+d e_\ell^*$ for some $\ell\not=j, k$ and $d\in \{\pm 1\}$, and hence $a_{\gamma,\lambda}=(\gamma,\lambda)=-1$.  
\end{proof}

The following is our main result in this section. 

\begin{theo} \label{theo:2-1}
Let $\Phi$ be an irreducible subsystem of $R(\MV)$ of $\rank\ge 2$. Then $\Phi$ is of type A, B or D. More precisely, we have 
\begin{enumerate}
\item $\Phi$ is of type B $\Longleftrightarrow$ $\Phi$ contains a simple root of type 1, 
\item $\Phi$ is of type D $\Longleftrightarrow$ any simple root in $\Phi$ is of type 2 and $\Phi$ contains conjugate simple roots,
\item $\Phi$ is of type A $\Longleftrightarrow$ any simple root in $\Phi$ is of type 2 and $\Phi$ does not contain conjugate simple roots,
\end{enumerate} 
where $\rank\Phi\ge 4$ when we say $\Phi$ is of type D in {\rm (2)} above.
\end{theo}

\begin{proof}
By \eqref{eq:2-4}, a root of type 1 is a short root while a root of type 2 is a long root, and the ratio of their length is $1:\sqrt{2}$ so that $\Phi$ is not of type ${\rm G_2}$.   

Suppose that $\Phi$ contains a simple root of type 1 (short root).  If $\Phi$ contains two simple roots $\alpha,\beta$ of type 1, then $a_{\beta,\alpha}=2(\beta,\alpha)=0$.  The reason is as follows.  Since $\alpha$ and $\beta$ are of type 1, $\alpha=\pm e_i^*$ and $\beta=\pm e_j^*$ for some $i$ and $j$, and since $\alpha$ and $\beta$ are simple roots, they are linearly independent; so $i\not=j$ and hence $(\alpha,\beta)=0$.  Therefore $\Phi$ must contain a simple root of type 2 (a long root) and any two short simple roots are not joined in the Dynkin diagram of $\Phi$.  Such $\Phi$ is of type B.  

By the above argument, we may assume that every simple root in $\Phi$ is of type 2 in the sequel. 
Therefore, all simple roots of $\Phi$ have the same length and such $\Phi$ must be of type A, D or E.  We shall prove that type E does not occur.  

Suppose that $\Phi$ is of type E.  Then the Dynkin diagram of $\Phi$ has a vertex connected to three other vertices and this means that $\Phi$ contains simple roots $\alpha,\beta,\gamma,\delta$ which satisfy the assumption in Lemma~\ref{lemm:2-2} (1).  Therefore, there is a conjugate pair in $\beta,\gamma,\delta$ and we may assume that $\beta$ and $\gamma$ are conjugate.  
Since the Dynkin diagram of $\Phi$ is of type E, either $\beta$ or $\gamma$, say $\beta$, is joined to some other simple root $\lambda$ in $\Phi$ so that $a_{\beta,\lambda}=-1$.  Then $a_{\gamma,\lambda}=-1$ by Lemma~\ref{lemm:2-2} (2) . This means that $\gamma$ and $\lambda$ are joined in the Dynkin diagram of $\Phi$ but this does not occur in the Dynkin diagram of type E.  Hence, $\Phi$ is not of type E. 

In the sequel, $\Phi$ is of type A or D.  We note that $-1$ does not appear more than twice in any row of the Cartan matrix of type A and that $-1$ appears three times in some row of the Cartan matrix of type D (see \cite[p.59]{hump72}).  Therefore, if $\Phi$ is of type D, then $\Phi$ contains conjugate simple roots by Lemma~\ref{lemm:2-2} (1). 

Conversely, suppose that $\Phi$ contains conjugate simple roots and $\rank\Phi\ge 4$.  Let $\beta$ and $\gamma$ be conjugate simple roots. Let $\lambda$ to be a simple root joined to $\beta$. Then since $a_{\beta,\lambda}=-1$, $a_{\gamma,\lambda}=-1$ by Lemma~\ref{lemm:2-2} (2).  This means that $\lambda$ is also joined to $\gamma$. But this does not occur in the Dynkin diagram of type A since $\rank\Phi\ge 4$, thus $\Phi$ is of type D. 
\end{proof}

If $\beta$ and $\gamma$ in $R(\MV)$ are conjugate, 
then $(\beta\pm\gamma)/2$ are in $R(\MV)$ and of type 1. Therefore, the following follows from Theorem~\ref{theo:2-1}. 

\begin{coro} \label{coro:2-1} 
Any irreducible factor of $R(\MV)$ is of type A or B.  If there is no element of type 1 in $R(\MV)$, then any irreducible factor of $R(\MV)$ is of type A.
\end{coro}

We conclude this section with a remark.  The dual of the root system $R(\MV)$ is given by 
\[
R^\vee(\MV):=\{ \alpha^\vee:=\frac{2\alpha}{(\alpha,\alpha)}\mid \alpha\in R(\MV)\}
\]
(see \cite[p. 43]{hump72}) and it is again a root system.  We note that 
\begin{enumerate}
\item  if $\alpha$ is of type 1, then $|\langle\alpha^\vee,v_i\rangle|=2$ for some $i$ and $\langle\alpha^\vee,v_k\rangle=0$ for $k\not=i$,
\item  if $\alpha$ is of type 2, then $|\langle \alpha^\vee,v_i\rangle|=|\langle \alpha^\vee,v_j\rangle|=1$ for some $i,j$ and $\langle\alpha^\vee,v_k\rangle=0$ for $k\not=i,j$. 
\end{enumerate}
Since the dual of an irreducible root system of type A (resp. type B) is of type A (resp. type C), any irreducible factor of $R^\vee(\MV)$ is of type A or C by Corollary~\ref{coro:2-1}.  

\section{Torus manifolds} \label{sect:3}

A \emph{torus manifold} $M$ is a closed orientable smooth $2n$-dimensional manifold 
with a smooth effective action of a compact torus $T$ of dimension $n$ having a fixed point.   
A closed codimension 2 submanifold of $M$ is called a \emph{characteristic submanifold} of $M$ if it is fixed pointwise under some circle subgroup of $T$ and has a $T$-fixed point.  There are only finitely many $T$-fixed points and characteristic submanifolds in $M$.  We denote the characteristic submanifolds of $M$ by $M_i$'s $(i\in [m]=\{1,\dots,m\})$.  Since $M$ is assumed to be orientable and $M_i$ is fixed pointwise under some circle subgroup, $M_i$ is also orientable.  We choose and fix an \emph{omniorientation} on $M$, which is an orientation on $M$ and on each $M_i$.

\subsection{Subfamilies of torus manifolds}

A toric variety $X$ (over the complex numbers $\C$) is a normal algebraic variety of complex dimension $n$ with an algebraic action of $(\C^*)^n$ having an open dense orbit, where $\C^*=\C\backslash \{0\}$. A toric variety is not necessarily compact and may have a singularity.  A compact smooth (in other words, complete non-singular) toric variety is often called a \emph{toric manifold}. A typical example of toric manifold is $\C P^n$ with a standard action of $(\C^*)^n$.  The $(\C^*)^n$-action on a toric manifold has a finitely many fixed points and a toric manifold with the restricted action of the compact torus $T$ is a torus manifold. 

A \emph{topological toric manifold} introduced in \cite{is-fu-ma13} is a topological analog of a toric manifold. It is a closed smooth manifold $M$ of dimension $2n$ with a smooth effective action of $(\C^*)^n$ such that it has an open dense orbit and $M$ is covered by finitely many invariant open subsets each equivariantly diffeomorphic to a faithful \emph{smooth} representation space of $(\C^*)^n$.  One can see that a toric manifold is a topological toric manifold and the family of topological toric manifolds is much wider than that of toric manifolds.  Similarly to the toric case, a topological toric manifold with the restricted action of the compact torus $T$ is a torus manifold. 

There is another topological analog of a toric manifold, now called a \emph{quasitoric manifold}, which was introduced by Davis-Januszkiewicz (\cite{da-ja91}).  The $T$-action on a torus manifold $M$ is called \emph{locally standard} if any point of $M$ has an invariant open neighborhood equivariantly diffeomorphic to an invariant open set of a faithful $T$-representation space of real dimension $2n$.  If the $T$-action on $M$ is locally standard, then the orbit space $M/T$ is a manifold with corners.  A quasitoric manifold is a locally standard torus manifold whose orbit space is a simple polytope of dimension $n$. 
It is not difficult to see that there are many quasitoric manifolds which do not arise from toric manifolds (\cite{da-ja91}) while it has been shown by Suyama \cite{suya14} recently that there are many toric manifolds which are not quasitoric.  
However, the family of topological toric manifolds contains both the family of toric manifolds and the family of quasitoric manifolds (\cite{is-fu-ma13}).  

The cohomology rings of topological manifolds are generated by degree two elements (\cite[Proposition 8.3]{is-fu-ma13}).  So, the sphere $S^{2n}$ with a standard action of $T$ is not topological toric when $n\ge 2$ while it is a torus manifold (also see \cite{kuro13}).  It is easy to see that there are many torus manifolds which have non-vanishing odd degree cohomology groups (see \cite[Section 11]{is-fu-ma13} for example), so the family of torus manifolds is much wider than that of topological toric manifolds.  

\subsection{Equivariant cohomology}

The equivariant cohomology of a torus manifold $M$ fits well to the study of $M$.  
It is defined to be 
\[
H^*_T(M):=H^*(ET\times_T M)
\]
where $ET\to BT$ is the universal principal $T$-bundle and $ET\times_T M$ is the orbit space 
of $ET\times M$ by the $T$-action given by $(e,p)\to (et^{-1},tp)$ for $(e,p)\in ET\times M$ and $t\in T$.  
Let $S=H^{>0}(BT)$ and 
\[
\H^*_T(M):=H^*_T(M)/S\text{-torsion}.
\]
If $H^{odd}(M)=0$ (this is the case when $M$ is a toric or quasitoric manifold), then 
$H^*_T(M)$ has no $S$-torsion and $\H^*_T(M)=H^*_T(M)$.  

We denote the equivariant Poincar\'e dual of $M_i$ by $\tau_i$.  It is an element of $H^2_T(M)$ but we also regard it as an element of $\H^2_T(M)$.    The ring $\H^*_T(M)$ also has a structure of an $H^*(BT)$-algebra, via the projection map $\pi\colon ET\times_T M\to BT$.  Since $H^*(BT)$ is generated by $H^2(BT)$ as a ring, the algebra structure over $H^*(BT)$ is determined by the image of $u\in H^2(BT)$ by $\pi^*\colon H^*(BT)\to \H^*_T(M)$ and it is described as follows. 

\begin{lemm}[Lemma 1.5 in \cite{masu99}] \label{lemm:3-1}
To each $i\in [m]$, there is a unique element $v_i\in H_2(BT)$ such that  
\begin{equation} \label{eq:3-1}
\pi^*(u)=\sum_{i=1}^m \langle u,v_i\rangle \tau_i \quad \text{in $\H^2_T(M)$ for any $u\in H^2(BT)$}
\end{equation}
where $\langle\ ,\ \rangle$ denotes the natural pairing between cohomology and homology.
\end{lemm}

If we change the omniorientation on $M$, then $\tau_i$ may become $-\tau_i$ and hence $v_i$ may become $-v_i$; so $\pm v_i$ is independent of the choice of an omniorientation. 
Note that there are exactly $n$ characteristic submanifolds meeting at each $T$-fixed point in $M$.  

\begin{lemm} [Lemma 1.7 in \cite{masu99}] \label{lemm:3-1-1}
Let $p$ be a $T$-fixed point of a torus manifold $M$ and let $I(p)$ be the subset of $[m]$ consisting of the subscripts of the $n$ characteristic submanifolds meeting at $p$.  Then $\{ v_i\}_{i\in I(p)}$ is a basis of $H_2(BT)$.  In particular, each $v_i$ is primitive and $\{v_i\}_{i=1}^m$ spans $H_2(BT)$.
\end{lemm}  

We shall explain that each $v_i$ has a nice geometrical meaning.  As is well-known, there is a canonical isomorphism 
\begin{equation} \label{eq:3-2}
H_2(BT)\cong \Hom(S^1,T),
\end{equation}
where $S^1$ denotes the unit circle of the complex numbers $\C$. For $v\in H_2(BT)$, we denote by $\lambda_v\in \Hom(S^1,T)$ the element corresponding to $v$ through the isomorphism \eqref{eq:3-2}. The omniorientation on $M$ induces an orientation on the normal bundle $\nu_i$ of $M_i$ so that $\nu_i$ can be regarded as a complex line bundle because $\nu_i$ is of real dimension 2.  With this understood, one can see that the element $v_i$ has the following two properties (see (1.8) and Lemma 1.10 in \cite{masu99}): 
\begin{enumerate}
\item[(P1)] $\lambda_{v_i}(S^1)$ is the circle subgroup of $T$ which fixes $M_i$ pointwise, 
\item[(P2)] the differential of $\lambda_{v_i}(z)$ for $z\in S^1(\subset \C)$ acts on $\nu_i$ as complex multiplication by $z$. 
\end{enumerate}
There are two primitive elements which satisfy (P1), and (P2) determines one of them; so (P1) and (P2) characterize $v_i$.  

\begin{defi} \label{defi:3-1}
Let $M$ be a torus manifold with an omniorientation and let $v_i$ $(i=1,\dots,m)$ be the elements of $H_2(BT)$ defined above.  Then the root system $R(M)(\subset H^{2}(BT))$ of $M$ is defined to be $R(\{v_i\}_{i=1}^m)$.  Note that $R(M)$ is independent of the choice of the omniorientation on $M$ because so is $\pm v_i$.  
\end{defi}

Here are two examples of $R(M)$ corresponding to Example~\ref{exam:2-1} in Section~\ref{sect:2}.  

\begin{exam} \label{exam:3-1}
(1) Take $M=S^{2n}$ (regarded as the unit sphere of $\C^n\oplus\R$) with the $T$-action given by 
$(z_1,\dots,z_n,y)\to (g_1z_1,\dots,g_nz_n,y)$ 
where $(g_1,\dots,g_n)\in (S^1)^n=T$.  The characteristic submanifolds are $M_i=\{ z_i=0\}$ for $1\le i\le n$ and the circle subgroup which fixes $M_i$ pointwise is $\{(1,\dots,1,g_i,1\dots,1)\in T\mid g_i\in S^1\}$ for each $i$, so that we may think of $\{v_i\}_{i=1}^n$ as $\{e_i\}_{i=1}^n$ of $\Z^n$ through an identification of $H_2(BT)=\Hom(S^1,T)$ with $\Z^n$.  Hence 
\[
R(S^{2n})=\{\pm e_i^*\ (1\le i\le n),\ \pm e_i^*\pm e_j^*\ (1\le i<j\le n) \}
\]
where $\pm e_i^*$ are of type 1 while $\pm e_i^*\pm e_j^*$ are of type 2 as remarked in Example~\ref{exam:2-1}.

(2) Take $M=\C P^{n}$ with the $T$-action given by $[z_1,\dots,z_n,z_{n+1}]\to [g_1z_1,\dots,g_nz_n,z_{n+1}]$.  The characteristic submanifolds are $M_i=\{ z_i=0\}$ for $1\le i\le n+1$ and the circle subgroup which fixes $M_i$ pointwise is $\{(1,\dots,1,g_i,1\dots,1)\in T\mid g_i\in S^1\}$ for $1\le i\le n$ and the diagonal subgroup of $T$ for $i=n+1$, so that we may think of $\{v_i\}_{i=1}^{n+1}$ as $v_i=e_i$ for $1\le i\le n$ and $v_{n+1}=\sum_{i=1}^ne_i$ through an identification of $H_2(BT)=\Hom(S^1,T)$ with $\Z^n$. Hence 
\[
R(\C P^n)=\{\pm e_i^*\ (1\le i\le n),\ \pm(e_i^*-e_j^*) \ (1\le i<j\le n)\}
\]
where any element in $R(\C P^n)$ is of type 2 as remarked in Example~\ref{exam:2-1}.
\end{exam}

The following lemma gives a necessary condition for $R(M)$ to have an element of type 1 or of type 2.  

\begin{lemm} \label{lemm:3-2}
Let $\alpha$ be an element of $R(M)$.  If $\alpha\in R(M)$ is of type 1, then $M_i^T=M^T$ for some $i$, and if $\alpha\in R(M)$ is of type 2, then $M_i^T\cup M_j^T=M^T$ for some distinct $i$ and $j$, where $X^T$ denotes the $T$-fixed point set in $X$.  
\end{lemm} 

\begin{proof}
If $\alpha\in R(M)$ is of type 1, then $\pi^*(\alpha)=\pm\tau_{i}$ by \eqref{eq:3-1}.  The restriction map $H^*_T(M)\to \bigoplus_{p\in M^T}H_T^*(p)$ sends $\pi^*(\alpha)$ to $\alpha$ in each component (note $H^*_T(p)=H^*(BT)$), in particular, non-zero in each component.  On the other hand, $\tau_{i}$ restricts to $0$ in $H^*_T(p)$ if $p\notin M_i^T$.  Therefore, $M_i^T=M^T$.  

Similarly, if $\alpha\in R(M)$ is of type 2, then $\pi^*(\alpha)=\pm \tau_{i}\pm\tau_{j}$ by \eqref{eq:3-1}, and this implies that $M_i^T\cup M_j^T=M^T$ by the same argument as above.  
\end{proof}

If a torus manifold $M$ is $S^{2n}$ $(n\ge 2)$ with the standard $T$-action or a product of $S^{2k}$ $(k\ge 2)$ and any torus manifold of dimension $2(n-k)$, then $M_i^T=M^T$ for some $i$ while if $H^*(M)$ is generated by $H^2(M)$ as a ring, then $M_i^T\not=M^T$ for any $i$ (see \cite{ma-pa06}).  The latter condition is satisfied when $M$ is a topological toric manifold (in particular, a toric manifold or a quasitoric manifold).  Therefore, the following follows from Corollary~\ref{coro:2-1} and Lemma~\ref{lemm:3-2}. 
\begin{coro} \label{coro:3-1}
If $H^*(M)$ is generated by $H^2(M)$ as a ring, then any irreducible subsystem of $R(M)$ is of type A.
\end{coro}

Here is an example of $R(M)$ corresponding to Example~\ref{exam:2-2}. 

\begin{exam} \label{exam:3-2}
Suppose that a torus manifold $M$ is $4$-dimensional.  Since a characteristic submanifold is of codimension two, orientable and is assumed to have a $T$-fixed point, every characteristic submanifold of our $M$ must be a $2$-sphere and have two $T$-fixed points.  Therefore, Lemma~\ref{lemm:3-2} tells us that $R(M)$ has no element of type 1 (resp. type 2) unless $|M^T|$ is $2$ (resp. $2,3,4$), and $R(M)$ is empty if $|M^T|\ge 5$, where $|M^T|$ denotes the cardinality of $M^T$.    
\end{exam}

\begin{rema}\label{rema:3-1}
For some class of manifolds with nice torus actions such as flag manifolds and toric hyperK\"ahler manifolds (also manifolds treated in \cite{kuro10-2}, \cite{kuro11-2}), one could define root systems similarly to Definition \ref{defi:3-1}. 
\end{rema}

\subsection{Weakly equivariant diffeomorphisms}

We will show in the next section that if the $T$-action on a torus manifold $M$ extends to an effective smooth action of a connected compact Lie group $G$, then the root system of $G$ is a subsystem of the root system $R(M)$ of $M$.   In this subsection, we prepare some necessary results to prove the fact. 

Let $\Aut(T)$ denote the group of automorphisms of $T$.  A diffeomorphism $\fg$ of a torus manifold $M$ is said to be \emph{weakly equivariant} if $\fg(tp)=\rho(t)\fg(p)$ for $t\in T$ and $p\in M$ with some $\rho\in\Aut(T)$, and in this case we often say that $\fg$ is \emph{$\rho$-equivariant}.  Note that $\fg$ is weakly equivariant if and only if $\fg$ is in the normalizer of $T$ in the group $\Diff(M)$ of diffeomorphisms of $M$. 

\begin{lemm} \label{lemm:3-3}
If $\fg\in \Diff(M)$ is $\rho$-equivariant, then $\fg$ permutes the characteristic submanifolds $M_i$'s, i.e., there is a permutation $\sigma$ on $[m]$ such that $\fg(M_i)=M_{\sigma(i)}$ for any $i\in [m]$.  Moreover, $\rho_*(v_i)=\epsilon_i v_{\sigma(i)}$ with some $\epsilon_i\in \{\pm 1\}$ for any $i\in [m]$, where $\rho_*$ is an automorphism of $H_2(BT)$ induced from $\rho$. 
\end{lemm}

\begin{proof}
Since $M_i$ is fixed pointwise by the circle subgroup $\lambda_{v_i}(S^1)$ and $\fg$ is $\rho$-equivariant, $\fg(M_i)$ is fixed pointwise by a circle subgroup $\rho(\lambda_{v_i}(S^1))$ which is $\lambda_{\rho_*(v_i)}(S^1)$.  This shows that $\fg(M_i)=M_j$ for some $j$ and $\rho_*(v_i)=v_j$ up to sign, proving the lemma. 
\end{proof}

Since $\rho$ induces a $\rho$-equivariant homeomorphism $E_\rho$ of $ET$ (with right $T$-action), a homeomorphism of $ET\times M$ sending $(e,p)$ to $(E_\rho(e), \fg(p))$ is $\rho$-equivariant so that it descends to a homeomorphism of $ET\times_T M$ and induces a graded ring automorphism $\fg^*$ of $H^*_T(M)$.  The differential $d\fg$ of $\fg$ maps $\nu_i$ to $\nu_{\sigma(i)}$ and it follows from Lemma~\ref{lemm:3-3} and the definition of $\tau_i$'s that 
\begin{equation} \label{eq:3-3}
\fg^*(\tau_{\sigma(i)})=\epsilon_i\tau_i.
\end{equation}
Note that $\epsilon_i=1$ if and only if $d\fg$ preserves the orientations on $\nu_i$ and $\nu_{\sigma(i)}$ induced from the omniorientation on $M$.

The following lemma, which is essentially due to Wiemeler \cite{wiem12}, will play a key role in our argument.  

\begin{lemm}[Lemma 2.1 in \cite{wiem12}] \label{lemm:3-4}
If the $\rho$-equivariant diffeomorphism $\fg$ of a torus manifold $M$ is in the identity component of $\Diff(M)$ and $\rho^*$ is a reflection on $H^2(BT)$, then the permutation $\sigma$ in Lemma~\ref{lemm:3-3} satisfies either of the following:
\begin{enumerate}
\item $\sigma$ is the identity, $\epsilon_i=-1$ for some $i$ and $\epsilon_j=1$ for $j\not=i$, 
\item $\sigma$ permutes two elements, say $i$ and $j$, fixes the others, $\epsilon_i=\epsilon_j$ and $\epsilon_k=1$ for $k\not=i,j$.
\end{enumerate}
Moreover, if $\fg$ preserves the omniorientation on $M$ (this is the case when $G$ preserves a stable complex structure on $M$ discussed in Section~\ref{sect:5}), then (1) above does not occur by the remark after \eqref{eq:3-3}.  
\end{lemm}

\begin{proof} 
We shall reproduce the proof in \cite{wiem12} with some modification for the reader's convenience.  

First we will treat the important case where $H^{odd}(M)=0$.  In this case, it is proved in \cite{masu99} that $H^2_T(M)$ is freely generated by $\tau_i$'s over $\Z$ and the following is exact:
\begin{equation} \label{eq:3-4}
0\to H^2(BT)\stackrel{\pi^*}\longrightarrow H^2_T(M)\stackrel{\iota^*}\longrightarrow H^2(M)
\end{equation}
where $\iota\colon M\to ET\times_TM$ is an inclusion of the fiber. Note that $\fg^*=\rho^*$ on $H^2(BT)$, which follows from the definition of $g^*$ on $H^*_T(M)$ mentioned after the proof of Lemma~\ref{lemm:3-3}. Moreover, $\rho^*$ is a reflection on $H^2(BT)$ by assumption and 
$\fg^*$ is the identity on $H^2(M)$ because $g$ is assumed to be in the identity component of $\Diff(M)$.  It follows from the exactness of \eqref{eq:3-4} that we have 
\begin{equation} \label{eq:3-5}
\begin{split}
\trace(\fg^*,H^2_T(M))&=\trace(\rho^*,H^2(BT))+\trace(\fg^*,\iota^*(H^2_T(M)))\\
&=\rank_\Z H^2(BT)-2+\rank_\Z \iota^*(H^2_T(M))\\
&=\rank_\Z H^2_T(M)-2.
\end{split}
\end{equation}
Since $H^2_T(M)$ is freely generated by $\tau_i$'s over $\Z$, the identity \eqref{eq:3-5} together with \eqref{eq:3-3} implies that either (1) or (2) in the lemma must occur. 

For a general torus manifold $M$, the argument above up to \eqref{eq:3-5} still holds but $H^2_T(M)$ may not be freely generated by $\tau_i$'s. Remember that $S=H^{>0}(BT)$ and $\H^*_T(M)=H^*_T(M)/S\text{-torsion}$. Since the subring $H^*(BT)$ of $H^*_T(M)$ is $S$-torsion free, we have 
\[
H^2(BT)\cap (S\text{-torsion in $H^2_T(M)$})=\{0\}.
\]
Therefore, it follows from the exactness of \eqref{eq:3-4} that the $S$-torsion in $H^2_T(M)$ injects to $H^2(M)$ by $\iota^*$.  Since $\fg^*$ is the identity on $H^2(M)$, this shows that $g^*$ is also the identity on the $S$-torsion.  Therefore, it follows from \eqref{eq:3-5} that we have 
\begin{equation} \label{eq:3-6}
\begin{split}
\trace(\fg^*,\H^2_T(M))&=\trace(\fg^*,H^2_T(M))-\rank_\Z(\text{$S$-torsion in $H^2_T(M)$})\\
&=\rank_\Z H^2_T(M)-2-\rank_\Z(\text{$S$-torsion in $H^2_T(M)$})\\
&=\rank_\Z \H^2_T(M)-2.
 \end{split}
 \end{equation}
 Since $\H_2^T(M)$ is freely generated by $\tau_i$'s over $\Z$ (\cite[Lemma 3.2]{masu99}), the identity \eqref{eq:3-6} together with \eqref{eq:3-3} implies the lemma.  
\end{proof}

\section{Root systems of extended compact Lie group actions} \label{sect:4}

In this section, we will prove that if a connected compact Lie group $G$ acts on a torus manifold $M$ extending the $T$-action, then the root system of $G$ is a subsystem of the root system $R(M)$ of $M$ (Theorem~\ref{theo:4-1}).  This fact gives an \lq\lq upper bound" for such $G$.

We recall a canonical isomorphism
\[
H^2(BT)\cong \Hom(T,S^1).
\]
For $\alpha\in H^2(BT)$, we denote by $\chi^\alpha\in \Hom(T,S^1)$ the element corresponding to $\alpha$ through the isomorphism. 

\begin{defi}
\label{defi:4-1}
A root of $G$ is a non-zero weight of the adjoint representation of $T$ on $\g\otimes\C$, where $\g$ denotes the Lie algebra of $G$.  We think of a root of $G$ as an element of $H^2(BT)$ (through the above canonical isomorphism) and denote the set of roots of $G$ by $\RG$.
\end{defi}

For $\alpha\in \RG$, we denote by $T_\alpha$ the identity component of the kernel of $\chi^\alpha$.  Let $\G0_\alpha$ be the identity component of the subgroup of $\G0$ which commutes with $T_\alpha$.  The group $N_{\G0_\alpha}(T)/T$ is of order two and the automorphism of $T$ defined by $t\to gtg^{-1}$ for $g\in N_{\G0_\alpha}(T)\backslash T$ does not depend on the choice of $g$, so we may denote it by $\rho_\alpha$. It is of order two, its fixed point set contains the codimension-one subtorus $T_\alpha$ and $\rho_\alpha^*(\alpha)=-\alpha$.  We note that $\rho_\alpha^*$ is the Weyl group action associated to $\alpha\in \RG$. 

Similarly, ${\rho_\alpha}_*$ is a reflection on $H_2(BT)$ and we note that
\begin{equation} \label{eq:4-1}
\text{$\Fix({\rho_\alpha}_*)=H_2(BT_\alpha)=\Ker\alpha$\quad in $H_2(BT)$}
\end{equation}
where $\alpha\in H^2(BT)$ is regarded as an element of $\Hom(H_2(BT),\Z)$. 

\begin{lemm} \label{lemm:4-1}
Either of the following holds for $\alpha\in\RG$:
\begin{enumerate}
\item There is $i\in [m]$ such that $\alpha$ takes a non-zero value on $v_i$ and zero on the others. 
\item There are $i, j\in [m]$ such that 
\begin{equation} \label{eq:4-2}
\langle \alpha,v_i\rangle=-\epsilon_\alpha\langle \alpha, v_j\rangle\not=0\ 
\text{and}\ \langle \alpha,v_k\rangle=0 \ \text{for any $k\not=i,j$}
\end{equation}
where $\epsilona=\pm 1$. 
\end{enumerate}
If an element of $N_{\G0_\alpha}(T)\backslash T$ preserves the omniorientation on $M$ (this is the case when $G$ preserves a stable complex structure on $M$ discussed in Section~\ref{sect:5}), then (1) above does not occur and $\epsilona=1$.  
\end{lemm}

\begin{proof}
Let $g\in N_{\G0_\alpha}(T)\backslash T$. We note that $g$ is a $\rho_\alpha$-equivariant diffeomorphism of $M$ and is in the identity component of $\Diff(M)$, so that we can apply Lemma~\ref{lemm:3-4} to our $g$. 

For case (1) of Lemma~\ref{lemm:3-4}, it follows from Lemma~\ref{lemm:3-3} that there is $i\in [m]$ such that 
\begin{equation} \label{eq:4-3}
{\rho_\alpha}_*(v_i)=-v_i\quad\text{and}\quad {\rho_\alpha}_*(v_k)=v_k\ \text{for $k\not=i$,}
\end{equation}
and this together with \eqref{eq:4-1} shows that (1) in our lemma occurs in this case. 

Similarly, for case (2) of Lemma~\ref{lemm:3-4}, it follows from Lemma~\ref{lemm:3-3} that there are distinct $i$ and $j$ such that 
\begin{equation} \label{eq:4-4}
{\rho_\alpha}_*(v_i)=\epsilon_i v_j,\quad {\rho_\alpha}_*(v_j)=\epsilon_j v_i, \quad 
{\rho_\alpha}_*(v_k)=v_k \ \text{for any $k\not=i,j$}
\end{equation}
and $\langle\alpha,v_k\rangle=0$ for $k\not=i,j$ by \eqref{eq:4-1}, where since $\epsilon_i=\epsilon_j$, we denote them by $\epsilona$.  Finally, since $\rho_\alpha^*(\alpha)=-\alpha$, we have
\[
\langle \alpha,v_i\rangle=-\langle \rho_\alpha^*(\alpha),v_i\rangle=-\langle \alpha,
{\rho_\alpha}_*(v_i)\rangle=-\epsilona\langle\alpha,v_j\rangle
\]
proving (2). 

The last statement in the lemma follows from the last statement in Lemma~\ref{lemm:3-4}.  
\end{proof}

We will say that $\alpha$ is of type 1 (resp. of type 2) if (1) (resp. (2)) in Lemma~\ref{lemm:4-1} occurs. 
Following \eqref{eq:2-4}, we define a positive definite symmetric bilinear form $(\ ,\ )$ on $H^2(BT)$ as follows:
\begin{equation} 
(\alpha,\beta):=\sum_{i=1}^m \langle \alpha, v_i\rangle\langle \beta,v_i\rangle\quad\text{for $\alpha,\beta\in H^2(BT)$}.
\end{equation}

\begin{lemm} \label{lemm:4-2}
Let $\alpha, \beta\in \RG$.  Then 
\begin{equation} \label{eq:4-5}
\rho_\alpha^*(\beta)=\beta-\frac{2(\beta,\alpha)}{(\alpha,\alpha)}\alpha.
\end{equation}
\end{lemm}

\begin{proof}
Since $\{v_k\}_{k=1}^m$ spans $H_2(BT)$, it suffices to check that the both sides at \eqref{eq:4-5} evaluated on $v_k$ agree for each $k\in [m]$.  

Suppose that $\alpha$ is of type 1. Then, it follows from \eqref{eq:4-3} that
\[
\langle \rho_\alpha^*(\beta),v_k\rangle=\langle\beta,{\rho_\alpha}_*(v_k)\rangle= \begin{cases} -\langle\beta,v_i\rangle \quad&\text{if $k=i$},\\
\langle\beta,v_k\rangle\quad&\text{if $k\not=i$},\end{cases}
\]
while since $\langle \alpha,v_k\rangle$ is non-zero only when $k=i$, the right hand side at \eqref{eq:4-5} evaluated on $v_k$ reduces to 
\[
\langle\beta,v_k\rangle-2\langle\beta,v_i\rangle= \begin{cases} -\langle\beta,v_i\rangle \quad&\text{if $k=i$},\\
\langle\beta,v_k\rangle\quad&\text{if $k\not=i$},\end{cases}
\]
proving \eqref{eq:4-5} when $\alpha$ is of type 1. 
When $\alpha$ is of type 2, the desired result can easily be checked using \eqref{eq:4-2} and \eqref{eq:4-4} and we leave the check to the reader.  
\end{proof}

We note that $2(\beta,\alpha)/(\alpha,\alpha)$ in \eqref{eq:4-5} is an integer because $\rho_\alpha^*$ is the Weyl group action associated to $\alpha$.  We denote $|\langle\alpha,v_i\rangle|$ ($=|\langle\alpha,v_j\rangle|$ when $\alpha$ is of type 2) by $N_\alpha$, and denote the length of a root $\alpha$, that is $\sqrt{(\alpha,\alpha)}$, by $\|\alpha\|$.  

\begin{lemm} \label{lemm:4-3}
Let $\alpha$ and $\beta$ be roots in an irreducible factor $\Phi$ of $\RG$ and assume $\alpha\not=\pm \beta$. Then the following holds.
\begin{enumerate}
\item If both $\alpha$ and $\beta$ are of type 1, then $(\alpha,\beta)=0$.
\item If both $\alpha$ and $\beta$ are of type 2 and $(\alpha,\beta)\not=0$, then $N_\alpha=N_\beta$ and $\|\alpha\|=\|\beta\|$.
\item Suppose that $\rank\Phi\ge 3$.  If $\alpha$ is of type 1 and $\beta$ is of type 2 and $(\alpha,\beta)\not=0$, then $N_\alpha=N_\beta$ and $\sqrt{2}\|\alpha\|=\|\beta\|$.
\end{enumerate} 
\end{lemm}

\begin{proof}
(1) It suffices to prove that there is no common $v_i$ on which both $\alpha$ and $\beta$ take non-zero values.  If there is such a common $v_i$, then $\alpha$ and $\beta$ take zero on all the $v_j$'s different from $v_i$ because they are of type 1, which means that $\alpha$ and $\beta$ are linearly dependent and hence $\alpha=\pm\beta$ since they are roots, but it is assumed that $\alpha\not=\pm\beta$.     

(2) Since $(\alpha,\beta)\not=0$ and $\alpha\not=\pm \beta$, there is only one common $v_i$ on which both $\alpha$ and $\beta$ take non-zero values.  Moreover, since $\alpha$ and $\beta$ are of type 2, $(\alpha,\alpha)=2\langle\alpha,v_i\rangle^2$ and $(\beta,\beta)=2\langle\beta,v_i\rangle^2$ by Lemma~\ref{lemm:4-1}.  Therefore, we have   
\[
\frac{2(\beta,\alpha)}{(\alpha,\alpha)}=\frac{\langle\beta,v_i\rangle}{\langle\alpha,v_i\rangle},\quad 
 \frac{2(\beta,\alpha)}{(\beta,\beta)}=\frac{\langle\alpha,v_i\rangle}{\langle\beta,v_i\rangle}.
 \]
Since these numbers are integers as remarked above, we have $N_\alpha=N_\beta$ and $\|\alpha\|=\|\beta\|$.  

(3) Since $\alpha$ is of type 1, there is a unique $v_i$ on which $\alpha$ takes a non-zero value.  Moreover, since $(\alpha,\beta)\not=0$, $\beta$ also takes a non-zero value on the $v_i$ and since $\beta$ is of type 2, $(\beta,\beta)=2\langle\beta,v_i\rangle^2$ by Lemma~\ref{lemm:4-1}.  Therefore we have 
\[
\frac{2(\beta,\alpha)}{(\alpha,\alpha)}=\frac{2\langle\beta,v_i\rangle}{\langle\alpha,v_i\rangle},\quad 
 \frac{2(\beta,\alpha)}{(\beta,\beta)}=\frac{\langle\alpha,v_i\rangle}{\langle\beta,v_i\rangle}.
 \]
Since these two numbers are integers, either 
\begin{enumerate}
\item[(a)] $N_\alpha=N_\beta$ and $\sqrt{2}\|\alpha\|=\|\beta\|$ or
\item[(b)] $N_\alpha=2N_\beta$ and $\|\alpha\|=\sqrt{2}\|\beta\|$
\end{enumerate}
must hold.  

Let $\alpha'$ and $\beta'$ be another pair of roots in $\Phi$ which are of type 1 and type 2 respectively and $(\alpha',\beta')\not=0$.  Since $\Phi$ is irreducible, there is a chain of roots $\beta_1(=\beta),\dots,\beta_m(=\beta')$ in $\Phi$ such that $(\beta_k,\beta_{k+1})\not=0$ for $k=1,2,\dots,m-1$.  By (1), we may assume that each $\beta_k$ is of type 2.  Then $\|\beta_1\|=\dots=\|\beta_m\|$ by (2).  Therefore, $\|\alpha\|=\|\alpha'\|$ because otherwise $\Phi$ contains roots with three different length which does not occur for an irreducible root system.  This shows that only one of (a) and (b) above occurs for $\Phi$.  

We shall prove that (b) does not occur when $\rank\Phi\ge 3$.  
Suppose that it occurs.  Then $\Phi$ must be of type C because an element of type 1 is a long root while an element of type 2 is a short root and long roots are not joined by (1).   Let $K$ be a simple Lie subgroup of $G$ corresponding to $\Phi$ and consider the $K$-orbit $Kp$ of a $T$-fixed point $p$.  The elements in $\Phi$ are homomorphisms from $T$ to $S^1$ and let $L$ be the identity component of the kernel of all those homomorphisms. Then $L$ is a torus of rank $n-\rank\Phi$.  Since $L$ commutes with $K$ by construction and $p$ is a $T$-fixed point (in particular, a $L$-fixed point), the orbit $Kp$ is fixed pointwise under the $L$-action.  Therefore, the tangent space $\tau_p(Kp)$ of $Kp$ at $p$ is fixed under the $L$-action, i.e. 
\begin{equation} \label{eq:4-6}
\tau_p(Kp)\subset (\tau_pM)^L 
\end{equation}
and moreover  
\begin{equation} \label{eq:4-7}
\dim (\tau_pM)^L\le 2\rank\Phi
\end{equation}
because the tangential $T$-representation $\tau_pM$ of $M$ at $p$ is faithful, $\dim \tau_pM=2\dim T$ and $\dim L=n-\rank\Phi$.  Now we consider two cases. 

Case 1. The case where $Kp=\{p\}$.  In this case, the $K$-representation $\tau_p(Kp)$ is faithful (because the $K$-action on $M$ is effective) so that $K$ must be embedded into a special orthogonal group $\SO(2\ell)$ by \eqref{eq:4-6} and \eqref{eq:4-7} where $\ell=\rank\Phi$.  However, this is impossible because $K$ is of type C and of $\rank\ell$.  

Case 2. The case where $Kp\not=\{p\}$.  
Let $L'$ be a maximal torus of $K$.  Then $\tau_p(Kp)^{L'}=0$.  Indeed, since $L'$ and $L$ generate the torus $T$, if $\tau_p(Kp)^{L'}\not=0$, then this together with \eqref{eq:4-6} implies that $(\tau_pM)^T\not=0$ but this contradicts the effectiveness of the $T$-action on $M$.  The fact $\tau_p(Kp)^{L'}=0$ implies that $\dim Kp\ge 2\dim L'=2\ell$ while $\dim Kp\le 2\ell$ by \eqref{eq:4-6} and \eqref{eq:4-7}.  This shows that $Kp$ is a torus manifold.  Since $Kp$ is a homogeneous space, the classification result in \cite{kuro10} for homogeneous torus manifolds shows that $K$ is not of type C.  Thus, case (b) does not occur when $\rank\Phi\ge 3$.   
\end{proof}

Note that $\RG$ lies in $H^2(BT)$ and the irreducible factors of $\RG$ are mutually perpendicular, i.e. if $\Phi$ and $\Phi'$ are distinct irreducible factors of $\RG$, then $(\gamma, \gamma')=0$ for any $\gamma\in\Phi$ and $\gamma'\in\Phi'$.  By Lemma~\ref{lemm:4-3}, if $\Phi$ is an irreducible factor of $\RG$, then $N_\gamma$ is independent of $\gamma\in\Phi$ unless $\Phi$ is of rank two and (b) in the proof of Lemma~\ref{lemm:4-3} occurs for $\Phi$.  When $\Phi$ is of rank two and (b) in the proof of Lemma~\ref{lemm:4-3} occurs for $\Phi$, the dual of $\Phi$, that is 
$\Phi^\vee:=\{ \alpha^\vee:=2\alpha/(\alpha,\alpha)\mid \alpha\in \Phi\}$, is isomorphic to $\Phi$ and $\Phi^\vee$ satisfies (a) in the proof of Lemma~\ref{lemm:4-3}, so $N_\gamma$ is independent of $\gamma\in \Phi^\vee$.  Therefore, in any case, we may assume that $N_\gamma$ is independent of $\gamma\in \Phi$ for each irreducible factor $\Phi$ of $\RG$.  Moreover, since we are concerned with the isomorphism type of $\RG$ as a root system, we may assume that $N_\gamma=1$ for any $\gamma\in \RG$ in the sequel.  Thus, the following theorem follows from Lemma~\ref{lemm:4-2} and the definition of $R(M)$.  

\begin{theo} \label{theo:4-1}
If a connected compact Lie group $G$ acts effectively on a torus manifold $M$ extending the $T$-action, then the root system $\RG$ of $G$ is a subsystem of the root system $R(M)$ of $M$. 
\end{theo} 

An effective action of a compact Lie group $G$ on $M$ can be viewed as a compact Lie subgroup of $\Diff(M)$ and vice versa.  Therefore, the following follows from Theorem~\ref{theo:4-1}. 

\begin{coro} \label{coro:4-1}
Let $M$ be a torus manifold and let $G$ be a compact connected Lie subgroup of $\Diff(M)$ containing the torus $T$.  If $\RG=R(M)$, then $G$ is maximal among connected compact Lie subgroups of $\Diff(M)$ containing $T$.  
\end{coro}

\begin{rema}
As pointed out by Wiemeler, the equality in Corollary~\ref{coro:4-1} cannot be attained in general.  For example, take an $n$-dimensional homology sphere $\Sigma$ with non-trivial fundamental group and consider $\Sigma\times T$, where the $T$-action on $\Sigma$ is trivial while that on the second factor $T$ is the group multiplication.  Take also the standard sphere $S^{2n}$ with the standard $T$-action.  We choose a free $T$-orbit from $\Sigma\times T$ and $S^{2n}$ respectively, remove their $T$-invariant open tubular neighborhoods and glue them  along their boundaries.  The resulting manifold $M$ is a homology sphere and a torus manifold.  The $v_i$'s for $M$ are the same as those for $S^{2n}$ because $v_i$'s are defined as parametrization of circle subgroups which fix the characteristic submanifolds and the characteristic submanifolds of $M$ and $S^{2n}$ are same.  Therefore $R(M)=R(S^{2n})$ and if there is a connected compact Lie subgroup $G$ of $\Diff(M)$ with $\RG=R(M)(=R(S^{2n}))$, then $G$ is a simple Lie group of type A and the $G$-action on $M$ must be transitive.  Therefore, $M$ must be diffeomorphic to $S^{2n}$ which is a contradiction because $M$ is non-simply connected.  Therefore, there is no compact Lie subgroup $G$ of $\Diff(M)$ with $\RG=R(M)$.  
\end{rema}

Here are examples corresponding to Example~\ref{exam:3-1}.

\begin{exam} \label{exam:4-1}
(1) Take $S^{2n}$ with the $T$-action defined in Example~\ref{exam:3-1}.  The linear action of $\SO(2n+1)$ on $S^{2n}$ is effective, $\SO(2n+1)$ contains the torus $T$ and $\Delta(\SO(2n+1))=R(S^{2n})$.  Therefore, $\SO(2n+1)$ is a maximal connected compact Lie subgroup of $\Diff(S^{2n})$ containing $T$ by Corollary~\ref{coro:4-1}.

(2) Take $\C P^{n}$ with the $T$-action defined in Example~\ref{exam:3-1}.  The linear action of $\U(n+1)$ on $\C^{n+1}$ induces an effective action of  $\PU(n+1)$ containing the torus $T$ and $\Delta(\PU(2n+1))=R(\C P^{n})$.  Therefore, $\PU(n+1)$ is a maximal connected compact Lie subgroup of $\Diff(\C P^{n})$ containing $T$ by Corollary~\ref{coro:4-1}.
\end{exam}

\section{Stable complex transformations} \label{sect:5}

A $T$-invariant stable complex structure $J$ on a torus manifold $M$ is a $T$-invariant complex structure on $\tau M\oplus\R^k$ for some $k\ge0$, where $\tau M$ denotes the tangent bundle of $M$ and $\R^k$ denotes the product bundle of rank $k$ over $M$.  
Since $J$ is $T$-invariant and a characteristic submanifold $M_i$ is fixed pointwise under the circle subgroup $v_i(S^1)$, where $v_i\in \Hom(S^1,T)=H_2(BT)$ as before, the quotient bundle of $\tau M\oplus\R^k$ restricted to $M_i$ by the tangent bundle of $M_i$ admits a complex structure induced from $J$.  We give a canonical orientation on $\R^k$.  Therefore, once we fix an orientation on $M$, $J$ determines an orientation on each characteristic submanifold $M_i$.  We note that if we change an orientation on $M$, then the induced orientation on $M_i$  also changes.  Therefore, the elements $v_i$ are independent of the choice of an orientation on $M$ and they are defined without ambiguity of sign.  

We denote by $(M,J)$ the torus manifold $M$ with the $T$-invariant stable complex structure $J$ and by $\Diff(M,J)$ the group of diffeomorphisms of $M$ which preserve $J$. The following fact, which we learned from Nigel Ray and thank him, shows that $\Diff(M,J)$ is much smaller than $\Diff(M)$. 

\begin{prop} \label{prop:5-1}
$\Diff(M,J)$ is a Lie group.
\end{prop}

\begin{proof}
By definition $J$ is a complex structure on $\tau M\oplus \R^k$ for some $k$.  Since the tangent bundle of  $\bar M:=M\times (S^1)^k$ is isomorphic to $\tau M\oplus\R^k$,  $J$ defines an almost complex structure $\bar J$ on $\bar M$,  so that $\Diff(\bar M,\bar J)$ is a Lie group (see \cite[Corollary 4.2 in p.19]{koba72}). Any diffeomorphism of $M$ extends to a diffeomorphism of $\bar M$ by defining the identity on the added factor $(S^1)^k$ so that $\Diff(M,J)$ can be viewed as a closed subgroup of $\Diff(\bar M,\bar J)$.  Since any closed subgroup of a Lie group is a Lie group as is well-known, the theorem follows.   
\end{proof}

Remember that the elements $v_i$'s $(1\le i\le m)$ in $H_2(BT)$ for $(M,J)$ are defined without ambiguity of sign.  Moreover, (1) in Lemma~\ref{lemm:4-1} does not occur in our stable complex case by the last statement in the lemma.  Motivated by this observation, we make the following definition.    

\begin{defi}
For $(M,J)$, we define 
\[
\begin{split}
R(M,J):=\{ \alpha\in H^2(BT)\mid &\ \langle \alpha,v_i\rangle=1, 
\langle \alpha,v_j\rangle=-1\  \text{for some $i\not=j$}\\
&\ \text{ and}\ \langle\alpha,v_k\rangle=0
\  \text{for any $k\not=i,j$}\}.
\end{split}
\]
\end{defi}

Note that $R(M,J)$ is a subsystem of $R(M)$ and there is no conjugate pair in $R(M,J)$.

\begin{theo} \label{theo:5-1}
Let  $M$ be a torus manifold with a $T$-invariant stable complex structure $J$.  If a connected compact Lie group $G$ acts effectively on $M$ extending the $T$-action and preserving the stable complex structure $J$, then the root system $\RG$ of $G$ is a subsystem of $R(M,J)$ and every irreducible factor of $\RG$ is of type A. 
\end{theo}

\begin{proof}
We may assume that any root $\alpha\in \RG$ lies in $R(M,J)$ by Lemma~\ref{lemm:4-1} and $\RG$ is a subsystem of $R(M)$ by Theorem~\ref{theo:4-1}, so $\RG$ is a subsystem of $R(M,J)$ and since there is no conjugate pair in $R(M,J)$, only type A appears as an irreducible factor of $\RG$ by Theorem~\ref{theo:2-1}. 
\end{proof}

\begin{rema}
It is shown in \cite[Lemma 5.8]{wiem10} that any irreducible factor of the $G$ in Theorem~\ref{theo:5-1} is of type A.  
\end{rema}

The following corollary is a stable complex version of Corollary~\ref{coro:4-1}. 

\begin{coro} \label{coro:5-1}
Let $M$ be a torus manifold with a stable complex structure $J$ and let $G$ be a connected compact Lie subgroup of $\Diff(M,J)$ containing the torus $T$.  If $\RG=R(M,J)$, then $G$ is maximal among connected compact Lie subgroups of $\Diff(M,J)$ containing $T$.  
\end{coro}

Here are two examples of $R(M,J)$ corresponding to Example~\ref{exam:4-1} in Section~\ref{sect:4}.   

\begin{exam} \label{exam:5-1}
(1) Take $M=S^{2n} (\subset \C^n\oplus\R)$ with the standard $T$-action defined in Example~\ref{exam:3-1}.  Since $\tau S^{2n}\oplus\R\cong \C^n\oplus\R$, we have 
\[
\tau S^{2n}\oplus\R\oplus\R\cong \C^n\oplus\R\oplus\R\cong \C^{n+1}.
\] 
Let $J$ be a stable complex structure defined from this identification.  Then one see that $\{v_i\}_{i=1}^n$ is the standard basis $\{e_i\}_{i=1}^n$ through an identification of $H_2(BT)=\Hom(S^1,T)$ with $\Z^n$, so that 
\[
R(S^{2n},J)=\{\pm(e_i^*-e_j^*)\ \ (1\le i<j\le n)\}.
\]  
This is a root system of type $A_{n-1}$.  Note that $R(S^{2n},J)$ is a proper subsystem of $R(S^{2n})$.  The linear action of $\U(n)$ on $\C^n$ induces an action on $S^{2n}$ preserving the $J$ and $\U(n)$ contains the torus $T$ and $\Delta(\U(n))=R(S^{2n},J)$.  Therefore, $\U(n)$ is a maximal connected compact Lie subgroup of $\Diff(S^{2n},J)$ containing $T$ by Corollary~\ref{coro:5-1}.  

(2) Take $M=\C P^n$ with the $T$-action defined in Example~\ref{exam:3-1} and with the standard complex structure $\Jst$.  Then one sees that 
$\{v_i\}_{i=1}^{n+1}$ as $v_i=e_i$ for $1\le i\le n$ and $v_{n+1}=-\sum_{i=1}^ne_i$ through an identification of $H_2(BT)=\Hom(S^1,T)$ with $\Z^n$. Hence 
\[
R(\C P^n,\Jst)=R(\C P^n)=\{\pm e_i^*\ (1\le i\le n),\ \pm(e_i^*-e_j^*) \ (1\le i<j\le n)\}.
\]
Since the action of $\PU(n+1)$ mentioned in Example~\ref{exam:4-1} preserves $\Jst$, $\PU(n+1)$ is also a maximal connected compact Lie subgroup of $\Diff(\C P^n,\Jst)$. 
\end{exam}

\section{Toric manifolds} \label{sect:6}

The complex projective space $\C P^n$ is a typical example of a complete non-singular toric variety (which we call a \emph{toric manifold}).  It turns out that $R(X,\Jst)=R(X)$ for any toric manifold $X$ with the standard complex structure $\Jst$ and that there is a connected compact Lie group $G$ preserving the $\Jst$ such that $\RG=R(X,\Jst)$, which we will discuss in this section (it is also discussed in \cite{masu10} from the viewpoint of symplectic category).  We also discuss $R(X,J)$ for an invariant stable complex structure $J$ different from the standard complex structure $\Jst$.  See \cite{fult93} or \cite{oda88} for the basics of toric varieties. 

\subsection{Fan and toric manifold} 
Let $\Delta$ be a complete non-singular fan of dimension $n$ in $H_2(BT)\otimes\R=H_2(BT;\R)$ and let $\Sigma$ be the underlying simplicial complex of $\Delta$.  Let $m$ be the number of vertices in $\Sigma$ and let $\{v_i\}_{i=1}^m$ be the primitive vectors on the one-dimensional cones in $\Delta$.  Since $\Delta$ can be recovered from $\Sigma$ and $\{v_i\}_{i=1}^m$, we may think of $\Delta$ as a pair $(\Sigma,\{v_i\}_{i=1}^m)$. 

Let $X(\Delta)$ be the toric manifold associated with the complete nonsingular fan $\Delta=(\Sigma,\{v_i\}_{i=1}^m)$.  
We shall recall the quotient construction of $X(\Delta)$.  
For $\sigma\subset [m]$ we set 
\[
L_\sigma:=\{(z_1,\dots,z_m)\in\C^m\mid z_i=0\ \text{for $i\in\sigma$}\}\quad\text{and}\quad Z:=\bigcup_{\sigma\notin\Sigma}L_\sigma.
\]
We define a homomorphism $\mathcal V\colon (\C^*)^m\to (\C^*)^n$ by 
\[
\mathcal V(g_1,\dots,g_m):=\prod_{i=1}^m\lambda_{v_i}(g_i)
\]
where $\lambda_{v_i}$ is a homomorphism from $\C^*$ to $(\C^*)^n$ determined by $v_i$.  
Then 
\[
X(\Delta)=(\C^m\backslash Z)/\Ker\mathcal V.
\]
The standard action of $(\C^*)^m$ on $\C^m$ induces an action of $(\C^*)^m/\Ker\mathcal V$ on $X(\Delta)$.  The homomorphism $\mathcal V$ is surjective and identifies $(\C^*)^m/\Ker\mathcal V$ with $(\C^*)^n$ so that $(\C^*)^n$ acts on $X(\Delta)$.  
The toric manifold $X(\Delta)$ with the restricted action of $T=(S^1)^n$ is a torus manifold.   

The projection $(\C^*)^m\to \C^*$ on the $i$-th factor induces a homomorphism $\rho_i\colon \Ker\mathcal V\to \C^*$.  Let $L_i$ be the complex line bundle over $X(\Delta)$ associated to $\rho_i$, namely the total space of $L_i$ is the orbit space of $(\C^m\backslash Z)\times\C$ by the diagonal action of $\Ker\mathcal V$ (through $\rho_i$ on $\C$) and the projection map on $X(\Delta)$ is the induced one from the projection on the first factor $\C^m\backslash Z$.  Then it is known and not difficult to see that there is a canonical isomorphism 
\begin{equation} \label{eq:6-1}
\tau X(\Delta)\oplus \underline{\C}^{m-n}\cong \bigoplus_{i=1}^mL_i \quad\text{as complex $T$-vector bundles},
\end{equation}
where $\underline{\C}^{m-n}$ denotes the product vector bundle $X(\Delta)\times\C^{m-n}$ (see \cite[Theorem 6.5]{is-fu-ma13} for example).    

\subsection{$R(X(\Delta),\Jst)$ for the standard complex structure $\Jst$} 
Let $\Jst$ be the standard complex structure on $X(\Delta)$.  For $\alpha\in R(X(\Delta),\Jst)$, there are $i,j\in [m]$ such that 
\[
\langle \alpha,v_i\rangle=1,\quad \langle \alpha,v_j\rangle=-1\quad 
\text{and}\quad \langle\alpha,v_k\rangle=0 \text{ for $k\not=i,j$.}
\]
Let $\chi^\alpha$ be the character of $(\C^*)^n$ corresponding to $\alpha$.  Since 
\[
\chi^\alpha(\prod_{i=1}^m\lambda_{v_i}(g_i))=
\prod_{i=1}^mg_i^{\langle\alpha,v_i\rangle}=g_ig_j^{-1},
\]
$\Ker (\chi^{\alpha}\circ\mathcal V)$ is a subgroup of $(\C^*)^m$ with $g_i=g_j$, which we denote by $H_\alpha$.  Obviously $H_\alpha$ contains $\Ker \mathcal V$. Let $\GL_\alpha$ be the (maximal) connected subgroup of $\GL(m,\C)$ which commutes with $H_\alpha$.  It is generated by $(\C^*)^m$ and linear automorphisms of $\C^m$ which fixes coordinates $z_k$'s for $k\not=i,j$.   

\begin{lemm} \label{lemm:6-1}
The action of $\GL_\alpha$ on $\C^m$ leaves $Z$ invariant. 
\end{lemm}

\begin{proof}
If $\sigma\subset\sigma'$, then $L_\sigma\supset L_{\sigma'}$; so $Z$ is the union of $L_{\sigma}$'s for minimal $\sigma\notin \Sigma$.  Therefore it suffices to prove that if $\sigma$ is minimal among subsets of $[m]$ not contained in $\Sigma$, then $\sigma$ contains both $i$ and $j$ or contains neither $i$ nor $j$.  

Suppose that $\sigma\ni i$ but $\sigma\not\ni j$ (and we will deduce a contradiction).  Since $\sigma$ is minimal, $\sigma\backslash\{i\}\in \Sigma$. Let $\tau$ be an $(n-1)$-simplex in $\Sigma$ containing $\sigma\backslash\{i\}$.  The cone $\angle v_\tau$ spanned by $v_\ell$'s for $\ell \in \tau$ is of dimension $n$. Since $\sigma\notin \Sigma$, we have $\tau\not\ni i$.  Then $\tau\ni j$ because otherwise $\angle v_\tau$ is contained in $\Ker\alpha$ and this contradicts the fact that $\angle v_\tau$ is of dimension $n$. The cone $\angle v_{\tau\backslash\{j\}}$ is of dimension $n-1$ and contained in $\Ker\alpha$.  The cone spanned by $\angle v_{\tau\backslash\{j\}}$ and $v_i$ is not a member of $\Delta$ because otherwise $(\tau\backslash\{j\})\cup\{i\}\in \Sigma$ and since this simplex contains $\sigma$, i.e., $\sigma$ must be a member of $\Sigma$, which gives a contradiction. 

The hyperplane $\Ker\alpha$ splits $H_2(BT;\R)$ into two (open) half spaces and contains all $v_k$'s except $v_i$ and $v_j$, so we can say $v_{i}$-side and $v_{j}$-side. The fact that the cone spanned by $\angle v_{\tau\backslash\{j\}}$ and $v_i$ is not a member of $\Delta$ means that a point of $H_2(BT;\R)$ which is slightly above the cone $\angle v_{\tau\backslash\{j\}}$ (here \lq\lq above" means that it is on the $v_i$ side) cannot be contained in any cone of $\Delta$.  This contradicts the completeness of $\Delta$.  Therefore the claim is proven.   
\end{proof}

Set 
$$H:=\bigcap_{\alpha\in R(X(\Delta),\Jst)}H_\alpha$$ 
and let $\GL(\Delta)$ be the subgroup of $\GL(m,\C)$ generated by $\GL_\alpha$'s for $\alpha\in R(X(\Delta),\Jst)$. The group generated by ${\rho_\alpha}_*$'s acts on the set $\{v_i\}_{i=1}^m$ (see Section~\ref{sect:4}) and decomposes it into a union of orbits.  Looking at the induces of $v_i$'s, this orbit decomposition produces a partition of $[m]$, say $\mu_1,\dots,\mu_s$. Then 
\[
H=\{(g_1,\dots,g_m)\in(\C^*)^m\mid g_i=g_j \ \text{when $i,j\in\mu_r$ for some 
$r$}\}.
\]
For a subset $\mu$ of $[m]$ we denote by $\GL(\mu)$ the subgroup of $\GL(m,\C)$ which fixes coordinates $z_k$'s for $k\notin\mu$. Then $\GL(\Delta)=\prod_{r=1}^s\GL(\mu_r)$.  The root system of $\GL(\Delta)$ agrees with $R(X(\Delta),\Jst)$.  

By Lemma~\ref{lemm:6-1}, the action of $\GL(\Delta)$ on $\C^m$ leaves $Z$ invariant.  Since $\GL(\Delta)$ commutes with $H$ containing $\Ker\mathcal V$, the action of $\GL(\Delta)$ on $\C^m\backslash Z$ descends to an action of $G(\Delta):=\GL(\Delta)/\Ker \mathcal V$ on $X(\Delta)=(\C^m\backslash Z)/\Ker\mathcal V$.  Therefore a maximal connected compact Lie subgroup $G$ of $G(\Delta)$ acts on $X(\Delta)$ preserving the complex structure $\Jst$.  The root system $\RG$ of $G$ agrees with that of $G(\Delta)$ and also that of $\GL(\Delta)$. Therefore $\RG=R(X(\Delta),\Jst)$. 
Consequently, by Corollary \ref{coro:5-1}, we have the following proposition.
\begin{prop} \label{prop:6-1}
Let $(X(\Delta),\Jst)$ be a toric manifold with the standard complex structure. Then, the $G$ above is a maximal connected compact Lie subgroup of ${\rm Diff}(X(\Delta),\Jst)$ containing $T$.
\end{prop}

\begin{rema}
Demazure defines roots associated to the fan $\Delta$ in \cite{de70} (also see \cite{oda88}) and intoduces a root system which describes the reductive part of $\Aut(X(\Delta))$.  Proposition \ref{prop:6-1} follows from his result.
\end{rema}

\subsection{$R(X(\Delta),J)$ for a non-standard stable complex structure $J$} 
We may regard the standard complex structure $\Jst$ as the stable complex structure induced from the complex structure on the right side of \eqref{eq:6-1}.  Each $L_i$ has another complex structure which is complex conjugate to the original one.  We denote $L_i$ with such a complex structure by $\bar{L_i}$.  We take an integer $q$ between $1$ and $m-1$ and consider the invariant stable complex structure $J$ on $X(\Delta)$ induced from $\bigoplus_{i=1}^qL_i\oplus\bigoplus_{i=q+1}^m \bar{L_i}$ through the map \eqref{eq:6-1}. 
In terms of $\C^m$ above, we are taking the standard complex structure on the first $q$ factors $\C^q$ but the \lq\lq anti-complex structure" on the last $m-q$ factors $\C^{m-q}$.  The subset of $H_2(BT)$ associated to $J$ is $\{v_i\}_{i=1}^q\cup\{-v_i\}_{i=q+1}^m$.  

\medskip
\noindent
{\bf Claim.}  There is no $\alpha\in H^2(BT)$ such that 
\[
\begin{split}
&\langle \alpha,v_i\rangle=1\quad \text{for some $i$ such that $1\le i\le q$},\\
&\langle\alpha,-v_j\rangle=-1,\quad \text{for some $j$ such that $q+1\le j\le m$},\\ 
&\langle\alpha,v_k\rangle=0\quad \text{for $k\not=i,j$}.
\end{split}
\]

\begin{proof}
If there is such $\alpha$, then all $v_\ell$'s sit on the non-negative side of $\alpha$ but this contradicts the completeness of $\Delta$. 
\end{proof} 
 
The above claim implies that $R(X(\Delta),J)\subset R(X(\Delta),\Jst)$.  Moreover, if $\GL(\Delta)_J$ denotes the subgroup of $\GL(m,\C)$ generated by $\GL_\alpha$'s for $\alpha \in R(X(\Delta),J)$, then 
\[
\GL(\Delta)_J=\GL(\Delta)\cap\big(\GL(q,\C)\times\GL(m-q,\C)\big)
\]
so that $G(\Delta)_J:=\GL(\Delta)_J/\Ker\mathcal V$ is contained in $G(\Delta)$.  The restricted action of $G(\Delta)_J$ on $X(\Delta)$ preserves  the stable complex structure $J$.  The root system of a maximal connected compact subgroup $G_J$ of $G(\Delta)_J$ agrees with that of $G(\Delta)_J$ and that of $\GL(\Delta)_J$.  Therefore $\Delta(G_J)=R(X(\Delta),J)$. 
Consequently, by Corollary \ref{coro:5-1}, we get the following generalization of Proposition \ref{prop:6-1}.
\begin{prop} \label{prop:6-2}
Let $(X(\Delta),J)$ be a toric manifold with the stable complex structure $J$.  Then, the $G_J$ above is a maximal connected compact Lie subgroup of ${\rm Diff}(X(\Delta),J)$ containing $T$.
\end{prop}

\begin{exam}
We take $X(\Delta)=\C P^1$.  In this case, $m=2$ and take $q=1$.  Then $G(\Delta)_J=(\GL(1,\C)\times \GL(1,\C))/\GL(1,\C)$ is isomorphic to $\GL(1,\C)$.  Note that our stable complex structure $J$ on $\C P^1$ is isomorphic to that in Example~\ref{exam:5-1} (1) for $S^{2}$. Therefore, $\Diff(S^2,J)$ is non-compact.  
\end{exam}

\section*{Acknowledgments}
The authors would like to thank Nigel Ray for informing them of Proposition \ref{prop:5-1} and Michael Wiemeler for useful comments.

\end{document}